\documentclass{amsart}
\usepackage{amsmath}
\usepackage{amsfonts}
\usepackage{amsthm}
\usepackage{amssymb}
\usepackage{enumerate}
\usepackage[all]{xy}
\usepackage{cite}
\usepackage{mathrsfs}
\usepackage{color}
\usepackage{xcolor}
\usepackage{hyperref}
\usepackage{fancyhdr}
\hypersetup{
	colorlinks=true,
	anchorcolor=blue,
	linkcolor=blue,
	filecolor=blue,
	urlcolor=blue,	
	citecolor=blue,
	bookmarks=true,
	bookmarksopen=true,
	pdfborder=000
}
\numberwithin{equation}{section}
\theoremstyle{plain}
\newtheorem{thm}{Theorem}[section]

\newtheorem{proposition}[thm]{Proposition}

\newtheorem{lemma}[thm]{Lemma}
\newtheoremstyle{noparens}%
  {}{}%
  {\itshape}{}%
  {\bfseries}{.}%
  { }%
  {\thmname{#1}\thmnumber{ #2}\mdseries\thmnote{ #3}}
\theoremstyle{noparens}
\newtheorem{lemmaNoParens}[thm]{Lemma}
\newtheorem{thmNoParens}[thm]{Theorem}

\theoremstyle{definition}
\newtheorem{defn}[thm]{Definition}
\theoremstyle{remark}
\newtheorem{rmk}[thm]{Remark}

\makeatletter

\newcommand{\Rmnum}[1]{\expandafter\@slowromancap\romannumeral #1@}
\@namedef{subjclassname@2020}{\textup{2020} Mathematics Subject Classification}
\makeatother
\begin{document}

\title{The Gehring-Hayman type theorem on pseudoconvex domains of finite type in $\mathbb{C}^2$}
\author{Haichou Li\textsuperscript{1} $\&$ Xingsi Pu\textsuperscript{2,3} $\&$ Hongyu Wang\textsuperscript{4,5}}

\address{$1.$ College of Mathematics and informatics, South China Agricultural University, Guangzhou, 510640, China}
\address{$2.$ HLM, Academy of Mathematics and Systems Science, Chinese Academy of Sciences, Beijing, 100190, China}
\address{$3.$ School of Mathematical Sciences, University of Chinese Academy of Sciences, Beijing, 100049, China}
\address{$4.$ School of Science, Beijing University of Posts and Telecommunications, Beijing, 100876, China}
\address{$5.$ Key Laboratory of Mathematics and Information Networks (Beijing University of Posts and Telecommunications), Ministry of Education, China}
\email{hcl2016@scau.edu.cn,\:puxs@amss.ac.cn,\:wanghyu@bupt.edu.cn}

\subjclass[2020]{Primary 32T25}
\keywords{quasi-geodesic, Kobayashi distance, pseudoconvex domain, Gromov hyperbolicity}
\thanks{The first author was supported by NSFC (No. 12226318, No. 12226334); the third author was supported by the Young Scientist Program of the Ministry of Science and Technology of China (No. 2021YFA1002200) and NSFC (No.12201059).}

\begin{abstract}
In this paper, we obtain the Gehring-Hayman type theorem on smoothly bounded pseudoconvex domains of finite type in $\mathbb{C}^2$. As an application, we provide a quantitative comparison between global and local Kobayashi distances near a boundary point for these domains.
\end{abstract}
\maketitle

\section{\noindent{{\bf Introduction}}}
In the complex plane $\mathbb{C}$, the classical Gehring-Hayman Theorem on planar domains was originally established by Gehring-Hayman in their work \cite{62GH}. The theorem was stated as follows.
\begin{thmNoParens}
Let $\Omega\subsetneq\mathbb{C}$ be a simply connected planar domain, then there exists a constant $C>0$ such that for any $x,y\in\Omega$,
\[L([x,y])\leq CL(\gamma),\]
where $[x,y]$ is the hyperbolic geodesic connecting $x$ and $y$, and $\gamma\subset\Omega$ is any curve with end points $x$ and $y$, and $L(\gamma)$ denotes the Euclidean length of $\gamma$.
\end{thmNoParens}

Later, in the real space $\mathbb{R}^n$, $n\geq2$, Gehring-Osgood \cite{GO79} generalized the result to quasihyperbolic geodesics on uniform domains by considering the quasihyperbolic metrics. Moreover, it was used to characterize the Gromov hyperbolicity of domains in $\mathbb{R}^n$(see \cite{BHK01,BB03}).

On the other hand, in the complex space $\mathbb{C}^n$, $n\geq2$, Liu-Wang-Zhou \cite{2022Bi} extended this result to $m$-convex domains by considering the Kobayashi metrics. Subsequently, Nikolov-\"{O}kten \cite{NA22} obtained similar results for more general domains and improved the index for $m$-convex domains. Recently, in our work \cite{liu2023bi}, we proved that the index can be $\frac{1}{m}$. Furthermore, Kosi\'{n}ski-Nikolov-Thomas \cite{23KNT} showed that for bounded strongly pseudoconvex domains with $C^{2,\alpha}$-smooth boundary, the index is 1.
In this paper, we get the Gehring-Hayman type theorem for smoothly bounded pseudoconvex domains of finite type in $\mathbb{C}^2$. The main theorem is as follows.

\begin{thm}\label{lgh}
Let $\Omega\subset \mathbb{C}^2$ be is a smoothly bounded pseudoconvex domain of finite type $m$. Then there exists a constant $C=C(\lambda)>0$ such that for every $x,y\in\Omega$,
\begin{align}\label{GH}
L(\gamma)\leq C|x-y|^{\frac{1}{m}},
\end{align}
where $\gamma$ is a Kobayashi $\lambda$-quasi-geodesic connecting $x$ and $y$ with $\lambda\geq1$. Moreover, denoting $H_\gamma=\max_{z\in\gamma}\delta(z)$, there exists a constant $C'=C'(\lambda)>0$ such that
\begin{align}\label{Sep}
H_\gamma^{\frac{1}{m}}\geq C'L(\gamma).
\end{align}
\end{thm}
\begin{rmk}
(1) Observing that for $m$-convex domains that are Gromov hyperbolic, the proof of $\lambda$-quasi-geodesic in \cite[Theorem 3.3]{liu2023bi} can be substituted with the proof of Theorem \ref{lgh}.

(2) It is worth noting that, inequality (\ref{Sep}) is also valid for smoothly bounded convex domains of finite type in $\mathbb{C}^n$, and Wang has obtained a similar estimate in \cite[Lemma 5.6]{wang22}. Moreover, for bounded strongly pseudoconvex domains with $C^2$-smooth boundary, the exponents are reduced to $1/2$ when utilizing Gromov hyperbolicity. This observation extends Corollary 12 in \cite{NA22} and further demonstrates the validity of \cite[Proposition 3]{23KNO} under lower boundary regularity assumptions.
\end{rmk}

In order to prove the theorem, we first illustrate the estimate in \cite[Theorem 1.3]{liu2023bi} is valid for $\lambda$-quasi-geodesic when the domain is Gromov hyperbolic.

\begin{thm}\label{Fin}
Let $\Omega \subset \mathbb{C}^{n}(n\geq2)$ be a bounded Gromov hyperbolic domain respect to a complete Finsler distance $d_F$. Suppose that there exist constants $C_1,C_2,\alpha>0,1>\beta>0$ such that, for all $x,\:y\in\Omega$,
\[
 d_F(x,y)\geq \alpha\left|\log\left(\frac{\delta(x)}{\delta(y)}\right)\right|-C_1,
\]
and the Finsler metric
\[
F(z,X)\geq\frac{C_2\left|X\right|}{\delta(z)^{\beta}}, \:\:\: \forall z\in\Omega,\; 0\neq X\in\mathbb{C}^n.
\]
For each $x,y \in \Omega$, let $L(\gamma)$ denote the Euclidean length of a Finsler $\lambda$-quasi-geodesic $\gamma$ connecting $x$ and $y$ in $\Omega$.
Then there exists a constant $C=C(\lambda)>0$ such that
\[
d_F(x,y)\geq2\alpha\log\left(\frac{L(\gamma)^{\frac{1}{\beta}}}{\sqrt{\delta(x)\delta(y)}}\right)-C.
\]
\end{thm}

As an application, we provide a quantitative comparison between global and local Kobayashi distances near a boundary point on smoothly bounded pseudoconvex domains of finite type in $\mathbb{C}^2$. It partially extends the result in \cite[Corollary 1.5]{2023Sar}.

\begin{thm}\label{comp}
Let $\Omega\subset \mathbb{C}^2$ be a smoothly bounded pseudoconvex domain of finite type $m$. Suppose that $\Omega\cap U$ is connected in a neighborhood $U$ of $\xi\in\partial\Omega$.
Then there exists a neighborhood $V$ of $\xi$ with $V\subset\subset U$ and constants $C>0,A\geq1$ such that, for any $x,y\in\Omega\cap V$,
\begin{align}\label{loc1}
 d_{K_\Omega}(x,y)\leq d_{K_{\Omega\cap U}}(x,y)\leq d_{K_\Omega}(x,y)+C|x-y|^{\frac{2}{Am}},
\end{align}
and
\begin{align}\label{loc2}
1\leq\frac{d_{K_{\Omega\cap U}}(x,y)}{d_{K_\Omega}(x,y)}\leq 1+C\left(\delta_\Omega(x)\wedge\delta_\Omega(y)+|x-y|^{\frac{1}{m}}\right)^{\frac{2}{A}}\text{ for }x\neq y\in\Omega\cap V.
\end{align}
\end{thm}

\begin{rmk}
The quantitative localization of Kobayashi distance on complex domains was first derived by Nikolov-Thomas \cite[Theorem 1.1]{2022Qu} near a locally $\mathbb{C}$-convexifiable point of finite type on the boundary. And inspired by their results, we relaxed the boundary regularity to Dini-smooth in \cite[Theorem 1.9]{liu2023bi}. Here we actually expand the boundary convexity for the result in $\mathbb{C}^2$.
\end{rmk}

The paper is organized as follows. In Sect.\ref{sec2} we will recall some definitions and preliminary results. In Sect.\ref{sec3} we give the proof of Theorem \ref{Fin} and Theorem \ref{lgh}. In Sect.\ref{sec4} we will illustrate the proof of Theorem \ref{comp}.

\section{Preliminaries}\label{sec2}
\subsection{Notation}
(1) For $z \in \mathbb{C}^n$, let $|\cdot|$ denote the standard Euclidean norm, and let $\left|z_1-z_2\right|$ denote the standard Euclidean distance of $z_1, z_2 \in \mathbb{C}^n$.

(2) Given an open set $\Omega \subsetneq \mathbb{C}^n, x \in \Omega$ and $v \in \mathbb{C}^n$, denote
$$
\delta(x)\text{ or }\delta_{\Omega}(x)=\inf \{|x-\xi|: \xi \in \partial \Omega\}.
$$

(3) Recall that, for real numbers $a, b$, $a \vee b:=\max \{a, b\}$ and $a \wedge b:=\min \{a, b\}$.

(4) For any curve $\gamma$, denote its Euclidean length by $L(\gamma)$.

(5) For functions $f, g$, write $f \lesssim g$ if there exists $C>0$ such that
$f \leq C g$, and we write
\[f \asymp g\text{ if }f \lesssim g\text{ and }g \lesssim f.\]

(6) Recall that, for any $p \in \partial \Omega$, the complex tangent space is given by
$$H_p \partial \Omega=\left\{X \in \mathbb{C}^n:\langle\bar{\partial} r(p), X\rangle=0\right\},$$
where $$\bar{\partial} r(p)=\left(\frac{\partial r}{\partial \bar{z}_1}(p), \ldots, \frac{\partial r}{\partial \bar{z}_n}(p)\right).$$
Here the standard Hermitian product in $\mathbb{C}^n$ is $\langle X, Y\rangle=\sum_{k=1}^n X_k \overline{Y}_k$. Therefore, for any vector $0\neq X \in \mathbb{C}^{n}$ it has a unique orthogonal decomposition $X=X_{H}+X_{N}$ with $X_{H} \in H_{p} \partial \Omega$ and $X_{N} \in N_{p} \partial \Omega$. Here $N_p\partial\Omega$ is the complex one-dimensional subspace of $\mathbb{C}^n$ orthogonal to $H_p\partial\Omega$.

\subsection{The Finsler metric}
Given a domain $\Omega\subset\mathbb{C}^n, n\geq2$, a Finsler metric on $\Omega$ is an upper semi-continuous map $F: \Omega \times \mathbb{C}^n \rightarrow [0, +\infty)$ with $F(z ; t X)=|t| F(z ; X)$ for any $z \in \Omega, \: t \in \mathbb{C}$ and $X \in \mathbb{C}^n$. The distance function $d_F$ associated with $F$ is defined by
\begin{multline}
\notag d_F(x, y)=\inf \{F\text{-length}(\gamma): \gamma:[0,1] \rightarrow \Omega \text{ is a piecewise }C^1\text{-smooth curve}\\
 \text{with }\gamma(0)=x, \gamma(1)=y\},
\end{multline}
where
$$
F\text{-length}(\gamma)=\int_0^1 F(\gamma(t) ; \dot{\gamma}(t)) d t.
$$

A very important Finsler metric in several complex variables is the Kobayashi metric. For a domain $\Omega \subset \mathbb{C}^n$, the (infinitesimal) Kobayashi metric is defined by
$$
K_{\Omega}(x ; v)=\inf \left\{|\xi|: f \in \operatorname{Hol}(\mathbb{D}, \Omega) \text {, with } f(0)=x, d(f)_0(\xi)=v\right\} .
$$
For convenience, we denote by $d_{K_\Omega}$ the Kobayashi distance associated with the Kobayashi metric $K_\Omega$, and sometimes we may omit the subscript $\Omega$. The main property of the Kobayashi distance is that it is contracted by holomorphic maps. That is, if $f: \Omega_1 \rightarrow \Omega_2$ is a holomorphic map, then
$$
\forall z, w \in \Omega_1 \quad d_{K_{\Omega_2}}(f(z), f(w)) \leqslant d_{K_{\Omega_1}}(z, w) .
$$

Recall that a $C^1$-smooth boundary point $p$ of a domain $\Omega$ in $\mathbb{C}^n$ is said to be Dini-smooth, if the outer unit normal vector $\vec{n}$ to $\partial \Omega$ near $p$ is a Dini-continuous function. This means that there exists a neighborhood $U$ of $p$ with
$$
\int_0^1 \frac{\omega(t)}{t} d t<+\infty,
$$
where
$$
\omega(t)=\omega(\vec{n}, \partial \Omega \cap U, t):=\sup \left\{\left|\vec{n}(x)-\vec{n}(y)\right|:|x-y|<t, \:\: x, y \in \partial \Omega \cap U\right\}
$$
is the respective modulus of continuity. Note that Dini-smooth is a weaker condition than $C^{1, \epsilon}$-smooth. Here a Dini-smooth domain means that each boundary point of $\Omega$ is a Dini-smooth point.

Then we have the following upper bound of Kobayashi distance.
\begin{lemmaNoParens}[{\cite[Corollary 8]{2015Estimates}}]\label{up}
Let $\Omega$ be a Dini-smooth bounded domain in $\mathbb{C}^n$ and $x, y \in \Omega$.
Then there exists a constant $C>1+\sqrt{2}/2$ such that
$$
K_{\Omega}(x, y) \leq \log \left(1+\frac{C|x-y|}{\sqrt{\delta_{\Omega}(x) \delta_{\Omega}(y)}}\right).
$$
\end{lemmaNoParens}

\subsection{Catlin-type metric}\label{sec2.3}
Suppose that $\Omega =\{z\in\mathbb{C}^{2}:r(z)<0\}$ is a smoothly bounded pseudoconvex domain of finite D'Angelo type, and $\xi \in \partial \Omega$ is a point of type $m_\xi$. By using a rotation of the canonical coordinates, we can arrange that the normal direction to $\partial \Omega$ at $\xi$ is given by the $\operatorname{Im} z_1$-axis. Supposing that $\xi=0$, and by using Implicit function theorem, we obtain a local defining function of the form $r\left(z_1, z_2\right)=$ $\operatorname{Im} z_2-F\left(z_1, \operatorname{Re} z_2\right)$, where $F$ is a smooth function and $F(0)=0$. As $\frac{\partial r}{\partial z_2}(\xi) \neq 0$, in a neighborhood of $\xi$ we can define the following vector fields
$$
L_1:=\frac{\partial}{\partial z_1}-\left(\frac{\partial r}{\partial z_2}\right)^{-1} \frac{\partial r}{\partial z_1} \frac{\partial}{\partial z_2}, \text { and } L_2:=\frac{\partial}{\partial z_2}.
$$
Note that $L_1r\equiv0$ and $L_1,L_2$ form a basis of $T_z^{1,0}$ for all $z$ near $\xi$.
For any $j, k>0$, set
$$
\mathcal{L}_{j, k}(z):=\underbrace{L_{1} \ldots L_{1}}_{j-1 \text { times }} \underbrace{\bar{L}_{1} \ldots \bar{L}_{1}}_{k-1 \text { times }} \partial\bar{\partial}r(L_1,\bar{L}_1)(z).
$$

As $\xi \in \partial \Omega$ is a point of type $m_\xi$, from the results in \cite[Theorem 2.4]{1977A}, it follows that there exist $j_0, k_0$ with $j_0+k_0=m_\xi$ which satisfy
$$
\mathcal{L}_{j, k}(\xi) =0 \quad j+k<m_\xi,\text{ and }\mathcal{L}_{j_0, k_0}(\xi) \neq 0.
$$
Denote
$$
C_{l}^{\xi}(z)=\max \left\{\left|\mathcal{L}_{j, k}(z)\right|: j+k=l\right\}.
$$
Let $X=b_1L_{1}+b_2 L_2$ be a holomorphic tangent vector at $z$. Now define the {\it Catlin metric}
$$
M_{\xi}(z, X):=\frac{|b_2|}{|r(z)|}+|b_1| \sum_{l=2}^{m_\xi}\left(\frac{C_{l}^{\xi}(z)}{|r(z)|}\right)^{\frac{1}{l}}.
$$

Owing to the result of Catlin\cite[Theorem 1]{1989Estimates}, the Kobayashi metric is locally bi-Lipschitz to the Catlin metric.
\begin{thm}
Let $\Omega =\{z\in \mathbb{C}^{2}:r(z)<0\}$ be a smoothly bounded pseudoconvex domain. If $\xi \in \partial \Omega$ be a point of finite type $m_\xi$, then there exist a neighborhood $U$ of $\xi$ and a constant $C \geq 1$ such that
$$
\frac{1}{C} M_{\xi}(z, X) \leq K_{\Omega}(z, X) \leq C M_{\xi}(z, X)
$$
for each $z \in \Omega \cap U$ and $X \in \mathbb{C}^{2}$.
\end{thm}

For our convenience, we may make a small change to the form of Catlin metrics. If $X$ is a holomorphic tangent vector at $z$, then it has the unique orthogonal decomposition $X=X_{H}+X_{N}$ by $X_{H} \in H_{\pi{(z)}} \partial \Omega$ and $X_{N} \in N_{\pi{(z)}} \partial \Omega$.
Here $\pi{(z)}$ is the closest point projected to the boundary (see \cite[Lemma2.1]{2000Gromov}).
As $L_2$ may be not parallel to $X_N$, we set
$$
\widetilde{M}_{\xi}(z, X):=\frac{|X_N|}{\delta(z)}+|X_H| \sum_{l=2}^{m_\xi}\left(\frac{C_{l}^\xi(z)}{\delta(z)}\right)^{\frac{1}{l}}.
$$
Lemma 2.2 in \cite{liu2023bi} implies that ${M}_{\xi}(z, X)\asymp\widetilde{M}_{\xi}(z, X)$ in a neighborhood $U$ of $\xi$.

Now choose open neighborhoods $U_i$ of $\xi_i\in\partial\Omega$, $1\leq i\leq s$, which form a finite cover of $\partial\Omega$. There exists a small $\varepsilon>0$ such that the neighborhood
$$
N_{\varepsilon}(\partial \Omega):=\left\{z \in \Omega : \delta_{\Omega}(z)<\varepsilon\right\}\subset\bigcup_{i=1}^{s}U_i.
$$
Denote $I_z:=\{i:z\in \overline{U}_i\}$, and set
\[\widetilde{M}(z,X):=\max_{i\in I_z}\{
\widetilde{M}_{\xi_i}(z, X)\} \text { for } z\in \Omega\cap N_{\varepsilon}(\partial \Omega).
\]
Since it is upper semi-continuous, we can define a global Finsler metric in $\Omega$ by
$$\widetilde{K}(z,X):=K(z,X)S(z,X)$$
with the positive function $S(z,X)\asymp1$, and $\widetilde{K}(z,X)=\widetilde{M}(z, X)$ for $z\in\Omega\cap N_\varepsilon(\partial\Omega)$,
which implies that $\widetilde{K}(z,X)\asymp K(z,X)$.

Denoting $m=\max\{m_{\xi_i}\}$, for $z\in\Omega\cap N_\varepsilon(\partial\Omega)$ and $X \in \mathbb{C}^{2}$, it now follows that
\begin{align}\label{nk}
\frac{|X_N|}{\delta(z)}+\frac{1}{C}\frac{|X_H|}{\delta(z)^{\frac{1}{m}}}\leq\widetilde{K}(z,X)\leq\frac{|X_N|}{\delta(z)}+C\frac{|X_H|}{\delta(z)^{\frac{1}{2}}}
\end{align}
for some constant $C>0$.
We may call it a Catlin-type metric, and denote by $d_{\widetilde{K}}$ the distance associated to this metric.

By using the proof of \cite[Theorem 1.12]{liu2023bi}, we know Catlin-type distances satisfy the following estimates.
\begin{lemma}\label{dg}
Assume that $\Omega \subset \mathbb{C}^{2}$ is a smoothly bounded pseudoconvex domain of finite type. Then there exists a constant $C \geq 1$ such that
$$
\left|\log\left(\frac{\delta(x)}{\delta(y)}\right)\right|-C\leq d_{\widetilde{K}}(x,y) \leq 2\log \left(1+\frac{|x-y|}{\sqrt{\delta(x) \delta(y)}}\right)+C
$$
for each $x,y \in \Omega$.
\end{lemma}

\subsection{Gromov hyperbolicity}\label{gh}

In this section we will give some definitions and results about Gromov hyperbolicity. Refer to \cite{Metric99} for further details.

Let $(X, d)$ be a metric space. The $Gromov\text{ } product$ of two points $x, y \in X$ with respect to a base point $\omega \in X$ is defined by
$$(x|y)_\omega:=\frac{1}{2}\left(d(x, \omega)+d(y, \omega)-d(x, y)\right).$$

Recall that a metric space $(X, d)$ is a $geodesic$ space if any two distinct points $x, y\in X$ can be joined by a $geodesic \text{ }segment$.
Furthermore, a metric space $(X, d)$ is called $proper$ if every closed ball in $(X, d)$ is compact.
A proper geodesic metric space $X$ is called $Gromov \text{ }hyperbolic$ if there is a constant $\delta\geq0$ such that, for any $x, y, z, \omega \in X$,
\begin{align}\label{gro}
(x|y)_\omega \geq \min \left\{(x|z)_\omega,(z|y)_\omega\right\}-\delta.
\end{align}
An equivalent definition of Gromov hyperbolicity is that each geodesic triangle is $\delta$-thin for some $\delta>0$, i.e., each side lies in the $\delta$-neighborhood of the other sides.

\begin{defn}
Let $(X,d)$ be a metric space and $I\subset\mathbb{R}$ be an interval. For $\lambda\geq1$ and $\kappa\geq0$, a map $\gamma: I \rightarrow X$ is called
a $(\lambda, \kappa)$-quasi-geodesic if for all $s,t\in I$,
\[
\lambda^{-1} |t-s|-\kappa \leq d(\gamma(t),\gamma(s)) \leq \lambda |t-s|+\kappa.
\]
In particular, $\gamma$ is called a geodesic when $\lambda=1,\kappa=0$, and $\gamma$ is called a $\lambda$-quasi-geodesic when $\kappa=0$.
\end{defn}

\begin{thmNoParens}[{\cite[Part \Rmnum{3}: Theorem 1.7]{Metric99}}]\label{stab}
Suppose that $(X,d)$ is a $\delta$-Gromov hyperbolic geodesic space with $\delta>0$. And suppose that $\gamma$ is a $(\lambda,\varepsilon)$-quasi-geodesic in $X$ and $[p, q]$ is a corresponding geodesic segment joining the endpoints of $\gamma$. Then there exists a constant $R = R(\delta,\lambda,\varepsilon)$ which satisfies that the Hausdorff distance between $[p, q]$ and the image of $\gamma$ is less than $R$.
\end{thmNoParens}

\section{Estimate of the Finsler distance}\label{sec3}

In this section we first prove Theorem \ref{Fin}. For convenience, we restate it as follows.
\begin{thm}
Let $\Omega \subset \mathbb{C}^{n}(n\geq2)$ be a bounded Gromov hyperbolic domain respect to a complete Finsler distance $d_F$. Suppose that there exist constants $C_1,C_2,\alpha>0,1>\beta>0$ with
\begin{align}\label{slow}
 d_F(x,y)\geq \alpha\left|\log\left(\frac{\delta(x)}{\delta(y)}\right)\right|-C_1, \:\: \text{ for any }x,y\in\Omega
\end{align}
and the Finsler metric
\begin{align}\label{fmet}
F(z,X)\geq\frac{C_2\left|X\right|}{\delta(z)^{\beta}}, \:\: \text{ for any } z\in\Omega,\; 0\neq X\in\mathbb{C}^n.
\end{align}
For each $x,y \in \Omega$, let $L(\gamma)$ denote the Euclidean length of a Finsler $\lambda$-quasi-geodesic $\gamma$ connecting $x$ and $y$ in $\Omega$.
Then there exists a constant $C=C(\lambda)>0$ such that
\[
d_F(x,y)\geq2\alpha\log\left(\frac{L(\gamma)^{\frac{1}{\beta}}}{\sqrt{\delta(x)\delta(y)}}\right)-C.
\]
\end{thm}

\begin{proof}
Now for any $\lambda$-quasi-geodesic $\gamma:[0,1] \rightarrow \Omega$ with $\gamma(0)=x$ and $\gamma(1)=y$, let $\eta:[0,1] \rightarrow \Omega$ be a geodesic with $\eta(0)=x$ and $\eta(1)=y$. Define $H:=\max _{z \in \gamma} \delta(z)$. There exists $t_0 \in[0,1]$ with $H=\delta\left(\gamma(t_0)\right)$.
Considering the subcurves $\gamma_1=\gamma |_{[0, t_0]}$ and $\gamma_2=\gamma |_{[t_0, 1]}$, there are two possibilities:

If $H \geq L(\gamma)^{\frac{1}{\beta}}$, as $(\Omega,\: d_F)$ is Gromov hyperbolic with $\delta>0$, there exists a constant $R = R(\delta,\lambda)$ such that the Hausdorff distance between $\gamma$ and $\eta$ is less than $R$. Choosing $t'_0\in[0,1]$ such that $d_K(\gamma(t_0),\eta(t'_0))\leq R$,
then by (\ref{slow}) we have
$$H=\delta(\gamma(t_0))\asymp\delta(\eta(t'_0))$$
and
$$
d_F(x,y)=d_F(x,\eta(t'_0))+d_F(\eta(t'_0),y) \geq \alpha\log \left(\frac{\delta(\eta(t'_0))}{\delta(x)}\right)+\alpha\log \left(\frac{\delta(\eta(t'_0))}{\delta(y)}\right).
$$
Thus
$$
d_F(x,y) \geq 2\alpha\log \left(\frac{H}{\sqrt{\delta(x)\delta(y)}}\right)-C\geq 2\alpha\log \left(\frac{L(\gamma)^{\frac{1}{\beta}}}{\sqrt{\delta(x)\delta(y)}}\right)-C,
$$
which completes the proof.

The other possibility is $H<L(\gamma)^{\frac{1}{\beta}}$. Since $\delta(x) \leq H$, there exists $k \in \mathbb{N}_+$ with
$$2^{-\frac{k}{\beta}} H<\delta(x) \leq 2^{-\frac{k-1}{\beta}} H.$$

Then we shall consider the following three alternatives:

(a). Consider the curve $\gamma_1$ and define $0=s_0 \leq s_1<\cdots<s_k \leq t_0$ as follows,
$$s_j=\min \left\{s \in\left[0, t_0\right]: \delta(\gamma(s))=\frac{H}{2^{\frac{k-j}{\beta}}}\right\}, \:\:\: j=1, \ldots, k.$$
By denoting $x_j=\gamma\left(s_j\right), \: j=0, \ldots, k$, we have
\begin{align}\label{bound}
1 \leq \frac{\delta\left(x_j\right)}{\delta\left(x_{j-1}\right)} \leq 2^{\frac{1}{\beta}}.
\end{align}

In the first case we assume that there exists an index $l \in\{1, \ldots, k\}$ with
$$
L(\gamma|_{[s_{l-1}, s_l]})>\frac{1}{8} 2^{-(k-l)} L(\gamma).
$$
Then, for $t \in\left[s_{l-1}, s_l\right]$, we have
$$
\delta(\gamma(t)) \leq 2^{-\frac{k-l}{\beta}} H,
$$
which implies that
\begin{align}\label{est}
L_F\left(\gamma |_{\left[s_{l-1}, s_l\right]}\right)
& \gtrsim \int_{s_{l-1}}^{s_l}\frac{|\dot{\gamma}(t)|}{{\delta(\gamma(t))}^{\beta}} dt
\geq\frac{2^{k-l}}{H^{\beta}} \int_{s_{l-1}}^{s_l}|\dot{\gamma}(t)| dt\\
&=\frac{2^{k-l}L(\gamma |_{\left[s_{l-1}, s_l\right]})}{H^{\beta}}
\notag\gtrsim\frac{L(\gamma)}{H^{\beta}}.
\end{align}

For $t_1:=s_k \leq t_0$, we choose $t'_1\in[0,1]$ such that $d_F(\gamma(t_1),\eta(t'_1))\leq R$. Then we have $H=\delta(\gamma(t_1))\asymp\delta(\eta(t'_1))$.
Let $[\gamma(t_1),\eta(t'_1)]$ be a Finsler geodesic connecting $\gamma(t_1)$ and $\eta(t'_1)$.
As $(\Omega,\: d_F)$ is Gromov hyperbolic, there exists a constant $M=M(\delta,\lambda)$ such that $x_{l-1},x_l\in N_M(\eta|_{[0,t'_1]}\cup[\gamma(t_1),\eta(t'_1)])$. Here $N_M(A):=\{z\in\Omega:d_F(z,A)<M\}$ for a subset $A$.

(1) If $x_{l-1},x_l\in N_M(\eta|_{[0,t'_1]})$, choose $x'_{l-1},x'_l\in\eta|_{[0,t'_1]}$ such that $d_F(x_{l-1},x'_{l-1})<M$ and $d_F(x_l,x'_l)<M$. It means $\delta(x_{l-1})\asymp\delta(x'_{l-1})$ and $\delta(x_l)\asymp\delta(x'_l)$.
Denote $x'_{l-1}=\eta(s'_{l-1})$ and $x'_l=\eta(s'_l)$ for $s'_{l-1},s'_l\in[0,t'_1]$.
Without loss of generality, we may suppose that $s'_{l-1}\leq s'_l$.
As $$d_F(x'_{l-1},x'_l)\geq d_F(x_{l-1},x_l)-2M\geq\frac{1}{\lambda}L_F\left(\gamma |_{\left[s_{l-1}, s_l\right]}\right)-2M,$$
combining (\ref{bound}) we have
\begin{align}
\notag d_F(x,y)&=d_F(x,x'_{l-1})+d_F(x'_{l-1},x'_l)+d_F(x'_l,\eta(t'_1))+d_F(\eta(t'_1),y)\\
\notag&\geq 2\alpha\log \left(\frac{\delta(\eta(t'_1))}{\sqrt{\delta(x)\delta(y)}}\right)+\alpha\log \left(\frac{\delta(x'_{l-1})}{\delta(x'_{l})}\right)
+\frac{1}{\lambda}L_F\left(\gamma |_{\left[s_{l-1}, s_l\right]}\right)-2M\\
\notag&\geq 2\alpha\log \left(\frac{H}{\sqrt{\delta(x)\delta(y)}}\right)+C\frac{L(\gamma)}{H^{\beta}}-C.
\end{align}

(2) If $x_{l-1},x_l\in N_M([\gamma(t_1),\eta(t'_1)])$, we choose $x'_{l-1},x'_l\in[\gamma(t_1),\eta(t'_1)]$ such that $d_F(x_{l-1},x'_{l-1})<M$ and $d_F(x_l,x'_l)<M$. So
$$d_F(x_{l-1},x_l)\leq d_F(x_{l-1},x'_{l-1})+d_F(x'_{l-1},x'_l)+d_F(x'_l,x_l)<2M+R.$$
Then
\begin{align}
\notag d_F(x,y)&\geq d_F(x,\eta(t'_1))+d_F(\eta(t'_1),y)+d_F(x_{l-1},x_l)-2M-R\\
\notag&\geq 2\alpha\log \left(\frac{\delta(\eta(t'_1))}{\sqrt{\delta(x)\delta(y)}}\right)+\frac{1}{\lambda}L_F\left(\gamma |_{\left[s_{l-1}, s_l\right]}\right)-C\\
\notag&\geq 2\alpha\log \left(\frac{H}{\sqrt{\delta(x)\delta(y)}}\right)+C\frac{L(\gamma)}{H^{\beta}}-C.
\end{align}

(3) If $x_{l-1},x_l$ do not satisfy the previous two cases, without loss of generality, we may suppose that
$x_{l-1}\in N_M(\eta|_{[0,t'_1]})$ and $x_l\in N_M([\gamma(t_1),\eta(t'_1)])$.
Choosing $x'_{l-1}\in\eta|_{[0,t'_1]}$ and $x'_l\in[\gamma(t_1),\eta(t'_1)]$ such that $d_F(x_{l-1},x'_{l-1})<M$ and $d_F(x_l,x'_l)<M$, we have
\begin{align}
\notag d_F(x_{l-1},x_l)&\leq d_F(x_{l-1},x'_{l-1})+d_F(x'_{l-1},\eta(t'_1))+d_F(\eta(t'_1),x'_l)+d_F(x'_l,x_l)\\
\notag &<d_F(x'_{l-1},\eta(t'_1))+2M+R.
\end{align}
Then we have
\begin{align}
\notag d_F(x,y)&\geq d_F(x,x'_{l-1})+d_F(x'_{l-1},\eta(t'_1))+d_F(\eta(t'_1),y)+d_F(\eta(t'_1),x'_l)-R\\
\notag&\geq 2\alpha\log \left(\frac{\delta(\eta(t'_1))}{\sqrt{\delta(x)\delta(y)}}\right)+\alpha\log \left(\frac{\delta(x'_{l-1})}{\delta(x'_{l})}\right)+d_F(x_{l-1},x_l)-2M-2R\\
\notag&\geq 2\alpha\log \left(\frac{H}{\sqrt{\delta(x)\delta(y)}}\right)+C\frac{L(\gamma)}{H^{\beta}}-C.
\end{align}

(b). Now consider the curve $\gamma_2$ and define $1=s_0^{\ast} \geq s_1^{\ast}>\cdots>s_k^{\ast} \geq t_0$ as follows,
$$s_j^{\ast}=\max \left\{s \in\left[t_0,1\right]: \delta(\gamma(s))=\frac{H}{2^{\frac{k-j}{\beta}}}\right\}, \:\:\: j=1, \ldots, k^{\ast}.$$
By denoting $x_j^{\ast}=\gamma\left(s_j^{\ast}\right), \: j=0, \ldots, k^{\ast}$, we have
\[
1 \leq \frac{\delta\left(x_j^{\ast}\right)}{\delta\left(x_{j-1}^{\ast}\right)} \leq 2^{\frac{1}{\beta}}.
\]

The second alternative is that there exists an index $l \in\{1, \ldots, k^{\ast}\}$ with
$$
L(\gamma|_{[s_{l}^{\ast}, s_{l-1}^{\ast}]})>\frac{1}{8} 2^{-(k^{\ast}-l)} L(\gamma).
$$

By applying similar considerations to the curve $\gamma_2$ instead of $\gamma_1$,
we can find $t_2 \in\left[t_0, 1\right]$ such that
$$d_F(x,y)\geq 2\alpha\log\left(\frac{H}{\sqrt{\delta(x)\delta(y)}}\right)+C\frac{L(\gamma)}{H^{\beta}}-C.$$

(c). The third alternative is
$$
L(\gamma|_{[s_{j-1}, s_j]}) \leq \frac{1}{8} 2^{-(k-j)} L(\gamma), \:\:\: j=1, \ldots, k,
$$
and
$$
L(\gamma|_{[s_{j-1}^{\ast}, s_j^{\ast}]}) \leq \frac{1}{8} 2^{-(k^{\ast}-j)} L(\gamma), \:\:\: j=1, \ldots, k^{\ast}.
$$
Then
$$
L(\gamma|_{[0, t_1]})=\sum_{j=1}^k L(\gamma|_{[s_{j-1}, s_j]}) \leq \frac{1}{4} L(\gamma)\text{ and }
L(\gamma|_{[t_2, 1]})=\sum_{j=1}^{k^\ast} L(\gamma|_{[s_{j-1}^{\ast}, s_j^{\ast}]}) \leq \frac{1}{4} L(\gamma).
$$
follows. We have
$$
L(\gamma|_{[t_1, t_2]})=L(\gamma)-L(\gamma|_{[0, t_1]})-L(\gamma|_{[t_2, 1]}) \geq \frac{1}{2} L(\gamma).
$$
Then similar to inequality (\ref{est}), it follows that $d_F(\gamma(t_1),\gamma(t_2))\gtrsim\frac{L(\gamma)}{H^{\beta}}$.

Now we choose $t'_1,t'_2\in[0,1]$ such that  $d_F(\gamma(t_1),\eta(t'_1))<R$ and $d_F(\gamma(t_2),\eta(t'_2))<R$. It means that
$$H=\delta(\gamma(t_1))\asymp\delta(\eta(t'_1)),\;\;\;H=\delta(\gamma(t_2))\asymp\delta(\eta(t'_2)).$$
Without loss of generality, we may suppose that $t'_1\leq t'_2$.
Therefore, we now obtain
\begin{align}
\notag d_F(x,y)&\geq d_F(x,\eta(t'_1))+d_F(\eta(t'_1),\eta(t'_2))+d_F(\eta(t'_2),y)\\
\notag&\geq 2\alpha\log \left(\frac{\sqrt{\delta(\eta(t'_1))\delta(\eta(t'_2))}}{\sqrt{\delta(x)\delta(y)}}\right)+d_F(\gamma(t_1),\gamma(t_2))-2R\\
\notag&\geq 2\alpha\log \left(\frac{H}{\sqrt{\delta(x)\delta(y)}}\right)+C\frac{L(\gamma)}{H^{\beta}}-C,
\end{align}
which implies that the above estimate is true in any case.

Let
$$
f(t):=2\alpha\log\left(\frac{t}{\sqrt{\delta(x)\delta(y)}}\right)+C\frac{L(\gamma)}{t^{\beta}}.
$$
Through a simple calculation, we know the function $f$ gets its minimum value when
$$
t=\left(\frac{\beta C}{2\alpha}\right)^{\frac{1}{\beta}}L(\gamma)^{\frac{1}{\beta}}
$$
which gives the lower bound
$$
d_F(x,y) \geq 2\alpha\log \left(\frac{L(\gamma)^{\frac{1}{\beta}}}{\sqrt{\delta(x)\delta(y)}}\right)-C,
$$
which completes the proof.
\end{proof}

Here we would like to provide an additional estimate which can also yield the above result by considering the condition (\ref{slow}).
Although the index is not more precise by using this approach and it has no direct connection with other aspects, we include it for readers who are interested.

\begin{proposition}
Let $\Omega \subset \mathbb{C}^{n}(n\geq2)$ be a bounded domain and let $d_F$ be a complete Finsler distance. Suppose that there exist constants $C_1,C_2,\alpha>0,1>\beta>0$ with
\[
 d_F(z,z_0)\leq \alpha\log\left(\frac{C_1}{\delta(z)}\right), \:\: \text{ for any }z \text{ and a fix }z_0 \in\Omega,
\]
and the Finsler metric
\[
F(z,X)\geq\frac{C_2\left|X\right|}{\delta(z)^{\beta}}, \:\: \text{ for any } z\in\Omega,\; 0\neq X\in\mathbb{C}^n.
\]
For each $x,y \in \Omega$, let $L(\gamma)$ denote the Euclidean length of a Finsler $\lambda$-quasi-geodesic $\gamma$ connecting $x$ and $y$ in $\Omega$ and $H_\gamma:=\sup_{z\in\gamma} \delta(z)$. Then there exists a constant $C=C(\lambda)>0$ such that
\[
L(\gamma)<C\text{ and }H_\gamma^\beta\gtrsim\frac{L(\gamma)}{\log{\frac{2C}{L(\gamma)}}}.
\]
\end{proposition}

\begin{proof}
For a Finsler $\lambda$-quasi-geodesic $\gamma:[a,b]\rightarrow\Omega$ with $\gamma(a)=x$ and $\gamma(b)=y$, there exists $T\in[a,b]$ such that $H_\gamma=\delta(\gamma(T))$. Since
\[
\frac{1}{\lambda}|T-t|\leq d_F(\gamma(T),\gamma(t))\leq d_F(\gamma(T),z_0)+d_F(z_0,\gamma(t))
\leq \alpha\log\left(\frac{1}{\delta(\gamma(T))\delta(\gamma(t))}\right)+C,
\]
we have
\[
\delta(\gamma(t))\leq\sqrt{\delta(\gamma(T))\delta(\gamma(t))}\leq \exp\left(-\frac{|T-t|}{2\lambda\alpha}+C\right).
\]
As $\gamma$ is a $\lambda$-quasi-geodesic, it follows that $F(\gamma(t),\dot{\gamma}(t))\leq \lambda$ almost everywhere.
Then
\begin{align}
\notag L(\gamma)&=\int_a^b|\dot{\gamma}(t)|dt\leq\frac{\lambda}{C_2}\int_a^b\delta(\gamma(t))^{\beta}dt\\
\notag&\lesssim\int_{[a,b]\cap[T-M,T+M]}\delta(\gamma(t))^{\beta}dt+\int_{[a,b]\cap[T-M,T+M]^c}\exp\left(-\frac{\beta|T-t|}{2\lambda\alpha}+C\right)dt\\
\notag &\lesssim 2MH_{\gamma}^{\beta}+\frac{4\lambda\alpha}{\beta}\exp\left(-\frac{\beta M}{2\lambda\alpha}+C\right)
\leq CMH_{\gamma}^{\beta}+C\exp\left(-\frac{\beta M}{2\lambda\alpha}\right).
\end{align}
As $H_{\gamma}$ has an upper bound in $\Omega$, there exist a constant $\widetilde{C}>0$ such that $L(\gamma)<\widetilde{C}$. Choosing $C>\widetilde{C}$ and letting $M=\frac{2\lambda\alpha}{\beta}\log\frac{2C}{L(\gamma)}>0$, we have $C\exp\left(-\frac{\beta M}{2\lambda\alpha}\right)=\frac{L(\gamma)}{2}$.
Hence
\[
H_{\gamma}^{\beta}\geq\frac{L(\gamma)}{2CM}\gtrsim\frac{L(\gamma)}{\log\frac{2C}{L(\gamma)}},
\]
which completes the proof.
\end{proof}
\begin{rmk}
Note that in \cite[Lemma 3.1]{2022Bi}, Liu-Wang-Zhou obtained similar estimates for $m$-convex domains. Later, Nikolov-\"{O}kten \cite[Corollary 11]{NA22} generalized the result to strongly Goldilocks domains as defined by them. However, previous proofs are based on the division of curves. Here we give a proof in integral form.
\end{rmk}

Now we begin to prove the Gehring-Hayman type theorem for smoothly bounded pseudoconvex domains of finite type in $\mathbb{C}^2$.
\begin{thm}
Let $\Omega\subset \mathbb{C}^2$ be is a smoothly bounded pseudoconvex domain of finite type $m$. Then there exists a constant $C=C(\lambda)>0$ such that for every $x,y\in\Omega$
\[L(\gamma)\leq C|x-y|^{\frac{1}{m}},\]
where $\gamma$ is a Kobayashi $\lambda$-quasi-geodesic connecting $x$ and $y$ with $\lambda\geq1$. Moreover, denoting $H_\gamma=\max_{z\in\gamma}\delta(z)$, there exists a constant $C'=C'(\lambda)>0$ such that
\[H_\gamma^{\frac{1}{m}}\geq C'L(\gamma).\]
\end{thm}

\begin{proof}
Since Kobayashi distance is bi-Lipschitz to Catlin-type distance, the Kobayashi $\lambda$-quasi-geodesic $\gamma$ is Catlin-type $\lambda'$-quasi-geodesic. By Theorem \ref{Fin} and Lemma \ref{dg}, from inequality (\ref{nk}) we have
\[
2\log\left(\frac{L(\gamma)^{m}}{\sqrt{\delta(x)\delta(y)}}\right)-C\leq d_{\widetilde{K}}(x,y)\leq 2\log \left(1+\frac{|x-y|}{\sqrt{\delta(x) \delta(y)}}\right)+C.
\]
It follows that
\[L(\gamma)^m\lesssim |x-y|+\sqrt{\delta(x)\delta(y)}.\]
Hence if $\sqrt{\delta(x)\delta(y)}\leq|x-y|$, we have the desired estimation $L(\gamma)\lesssim |x-y|^{\frac{1}{m}}$.

When $\sqrt{\delta(x)\delta(y)}>|x-y|$, let $\eta$ be a Catlin-type geodesic connecting $x$ and $y$. Denoting $H:=\max_{z\in\eta}\delta(z)$, it follows from Lemma \ref{dg} that
\[2\log\left(\frac{H}{\sqrt{\delta(x)\delta(y)}}\right)\leq d_{\widetilde{K}}(x,y)\leq2\log \left(1+\frac{|x-y|}{\sqrt{\delta(x) \delta(y)}}\right)+C,\]
which implies $H\lesssim |x-y|+\sqrt{\delta(x)\delta(y)}$.

From the result of Fiacchi \cite{2022Gromov} or Li-Pu-Wang \cite{23LPW} recently, we know $(\Omega,\: d_{\widetilde{K}})$ is Gromov hyperbolic with $\delta>0$.
Applying Theorem \ref{stab}, there exists a constant $R=R(\delta,\lambda')$ such that the Hausdorff distance between $\eta$ and $\gamma$ is less than $R$. Denoting $H_\gamma:=\max_{\omega\in\gamma}\delta(\omega)=\delta(\omega_0)$, then there exists a point $z_0\in\eta$ with $d_{\widetilde{K}}(\omega_0,z_0)\leq R$.

Applying Lemma \ref{dg}, we have $\delta(\omega_0)\asymp\delta(z_0)$, which implies that
\[H_\gamma\lesssim\delta(z_0)\leq H\lesssim|x-y|+\sqrt{\delta(x)\delta(y)}.\]

By using Lemma \ref{up}, we get
$$
\begin{aligned}
\frac{L(\gamma)}{H_\gamma^{\frac{1}{m}}}&\leq \int_0^1\frac{|\dot{\gamma}(t)|}{\delta(\gamma(t))^{\frac{1}{m}}}dt\lesssim \int_0^1K(\gamma(t),\dot{\gamma}(t))dt= L_K(\gamma)\leq \lambda d_K(x,y)\\
&\lesssim\log\left(1+\frac{C|x-y|}{\sqrt{\delta(x)\delta(y)}}\right)\lesssim\frac{|x-y|}{\sqrt{\delta(x)\delta(y)}}.
\end{aligned}
$$
Then
\[
L(\gamma)\lesssim\frac{\left(|x-y|+\sqrt{\delta(x)\delta(y)}\right)^{\frac{1}{m}}|x-y|}{\sqrt{\delta(x) \delta(y)}}\lesssim|x-y|^{\frac{1}{m}}.
\]

Similary, from the Gromov hyperbolicity, it also has $H\lesssim H_\gamma$. Fix a point $w\in\Omega$. By using Theorem \ref{Fin} and Lemma \ref{dg}, we obtain that
\[\log\left(\frac{\delta(w)}{H}\right)-C\leq d_{\widetilde{K}}(w,\eta)\leq(x|y)_w^{\widetilde{K}}+2\delta\leq\log\left(\frac{C}{L(\gamma)^m}\right).\]
Hence $H_\gamma\gtrsim H\gtrsim L(\gamma)^m$, which completes the proof.
\end{proof}

\section{The proof of Theorem \ref{comp}}\label{sec4}

To prove Theorem \ref{comp}, we first require the localization result of Kobayashi metric. This is based on Royden's Localization Lemma \cite[Lemma 2]{royden2006remarks}, with its proof available in \cite[Lemma 4]{Graham75}. Noting $\tanh(x)\geq1-2e^{-2x}$ for $x\geq0$, the Lemma 3.1 in \cite{2023Sar} can be stated as follows.
\begin{lemma}\label{metric}
Let $\Omega\subset\mathbb{C}^n$ is a Kobayashi hyperbolic domain and $U$ is an open subset of $\mathbb{C}^n$ such that $U\cap\Omega\neq\emptyset$ and connected. Then for every $W\subset\subset U$ with $W\cap\Omega\neq\emptyset$ and $d_K(W\cap\Omega,\Omega\backslash U)>0$, there exists a constant $L>0$ such that
\[
K_\Omega(z,X)\leq K_{U\cap\Omega}(z,X)\leq\left(1+Le^{-2d_K(z,\Omega\backslash U)}\right)K_\Omega(z,X).
\]
for all $z\in W\cap\Omega$ and $X\in\mathbb{C}^n$.
\end{lemma}

\noindent$Proof\;of\;Theorem\;\ref{comp}$.
By Lemma \ref{metric}, it follows that there exist a neighborhood $V_0$ of $\xi$ with $V_0\subset\subset U$ and a constant $C>0$ such that, for $z\in\Omega\cap V_0$ and $X\in\mathbb{C}^n$,
\[
K_\Omega(z,X)\leq K_{\Omega\cap U}(z,X)\leq\left(1+Ce^{-2d_K(z,\Omega\backslash U)}\right)K_\Omega(z,X).
\]

Hence we only need to check the right side of inequality (\ref{loc1}) and (\ref{loc2}).
By using the lower bound in \cite[Theorem 1.12]{liu2023bi}, we have
\[
d_{\widetilde{K}}(z,\Omega\backslash U)\geq\min_{\omega\in\Omega\backslash U}2\log\left(\frac{{|z-\omega|}^{m}+\delta(z)\vee\delta(\omega)}{\sqrt{\delta(z)\delta(\omega)}}\right)-C\geq\log\frac{1}{\delta(z)}-C.
\]
Then there exists a constant $A\geq1$ such that
\[
d_K(z,\Omega\backslash U)\geq\frac{1}{A}d_{\widetilde{K}}(z,\Omega\backslash U)\geq\frac{1}{A}\log\frac{1}{\delta(z)}-C,
\]
which means that
\begin{align}\label{loc}
K_{\Omega\cap U}(z,X)\leq\left(1+C{\delta(z)}^{\frac{2}{A}}\right)K_\Omega(z,X).
\end{align}

For a Kobayashi geodesic $\gamma:[a,b]\rightarrow\Omega$ with $\gamma(a)=x$ and $\gamma(b)=y$, there exists $T\in[a,b]$ such that $H_{\gamma}=\delta(\gamma(T))=\max_{t\in[a,b]}\delta(\gamma(t))$. By using the proof of Theorem \ref{lgh}, it follows that
\[
H_{\gamma}\lesssim|x-y|+\sqrt{\delta(x)\delta(y)}\text{ and }L(\gamma)\lesssim|x-y|^\frac{1}{m} .
\]
As $L(\gamma)\rightarrow0$ when $|x-y|\rightarrow0$, we can choose a neighborhood $V$ of $\xi$ with $V\subset V_0$ such that, for any $x,y\in\Omega\cap V$ , the Kobayashi geodesic $\gamma$ connecting $x$ and $y$ is in $\Omega\cap V_0$.
Since
\begin{align}
\notag|T-t|&= d_K(\gamma(T),\gamma(t))\leq \log\left(1+\frac{C|x-y|}{\sqrt{\delta(\gamma(T))\delta(\gamma(t))}}\right)\\
\notag &\leq \log\left(\frac{H_{\gamma}+L(\gamma)}{\sqrt{\delta(\gamma(T))\delta(\gamma(t))}}\right)+C,
\end{align}
we have
\[
\delta(\gamma(t))\leq\sqrt{\delta(\gamma(T))\delta(\gamma(t))}\leq \left(H_{\gamma}+L(\gamma)\right)e^{-|T-t|}.
\]

Then for $x,y\in\Omega\cap V$, when $|x-y|\geq\sqrt{\delta(x)\delta(y)}$ we obtain that
\[
\int_a^b{\delta(\gamma(t))}^{\frac{2}{A}}K_\Omega(\gamma(t),\dot{\gamma}(t))dt
\lesssim\left(H_{\gamma}+L(\gamma)\right)^{\frac{2}{A}}\int_{\mathbb{R}}e^{-|T-t|}dt\lesssim|x-y|^{\frac{2}{Am}}.
\]
And when $|x-y|<\sqrt{\delta(x)\delta(y)}$, we have
\begin{align}
\notag&\int_a^b{\delta(\gamma(t))}^{\frac{2}{A}}K_\Omega(\gamma(t),\dot{\gamma}(t))dt
\lesssim H_{\gamma}^{\frac{2}{A}}L_K(\gamma)
\lesssim H_{\gamma}^{\frac{2}{A}}\log\left(1+\frac{C|x-y|}{\sqrt{\delta(x)\delta(y)}}\right)\\
\notag&\lesssim\frac{\left(|x-y|+\sqrt{\delta(x)\delta(y)}\right)^{\frac{2}{A}}|x-y|}{\sqrt{\delta(x) \delta(y)}}\lesssim|x-y|^{\frac{2}{A}}.
\end{align}
Hence
\[
d_{K_{\Omega\cap U}}(x,y)\leq L_{K}(\gamma)+C
\int_a^b{\delta(\gamma(t)}^{\frac{2}{A}}K_\Omega(\gamma(t),\dot{\gamma}(t))dt
\leq d_{K_\Omega}(x,y)+C|x-y|^{\frac{2}{Am}}.
\]
Additionally, for $z\in\gamma$ it follows that
$$
\delta_\Omega(z)\leq\delta_\Omega(x)\wedge\delta_\Omega(y)+L(\gamma)
\lesssim\delta_\Omega(x)\wedge\delta_\Omega(y)+|x-y|^{\frac{1}{m}}.
$$
Therefore, for $x\neq y\in\Omega\cap V$, by estimate (\ref{loc}) we deduce that
\[
\frac{d_{K_{\Omega\cap U}}(x,y)}{d_{K_\Omega}(x,y)}\leq 1+C\left(\delta_\Omega(x)\wedge\delta_\Omega(y)+|x-y|^{\frac{1}{m}}\right)^{\frac{2}{A}},
\]
which completes the proof.
$\hfill\qed$

\vspace{0.3cm} \noindent{\bf Acknowledgements}. The authors would like to thank Professor Jinsong Liu for many precious suggestions.

\bibliography{reference}
\bibliographystyle{plain}{}
\end{document}